\documentclass[12pt, a4paper, a4wide, reqno]{amsart}

\usepackage[utf8]{inputenc}

\usepackage{amssymb,latexsym}
\usepackage{amsfonts}
\usepackage{amsmath}

\usepackage{enumerate}
\usepackage{float}
\usepackage{mathtools}
\usepackage{enumerate}
\usepackage{booktabs}

\usepackage[usenames, dvipsnames]{xcolor}
\usepackage[colorlinks,linkcolor=blue,anchorcolor=blue,citecolor=blue]{hyperref}

\usepackage{graphicx}

\usepackage{caption}
\usepackage{subcaption}

\makeatletter
\@namedef{subjclassname@2020}{%
  \textup{2020} Mathematics Subject Classification}
\makeatother

\newcommand\al{\alpha}

\def\ind{1\hspace{-2,9pt}\textnormal{I}}

\newtheorem{theorem}{Theorem}

\newtheorem{corollary}[theorem]{Corollary}
\newtheorem{lemma}[theorem]{Lemma}

\newtheorem{example}[theorem]{Example}

\theoremstyle{definition}
\newtheorem{remark}[theorem]{Remark}

\title{Multi seasonal discrete time risk model revisited}

\thanks{The third author is supported by grant No.\ S-MIP-20-16 from the Research Council of Lithuania}

\author{Andrius Grigutis}
\email{andrius.grigutis@mif.vu.lt}

\author{Jonas Jankauskas}
\email{jonas.jankauskas@gmail.com}

\author{Jonas \v{S}iaulys}
\email{jonas.siaulys@mif.vu.lt}

\address{Institute of Mathematics, Vilnius University, Naugarduko 24, Vilnius LT-03225, Lithuania}



\begin{document}

\maketitle

\begin{abstract}
In this work we set up the distribution function of $\mathcal{M}:=\sup_{n\geqslant1}\sum_{i=1}^{n}{(Z_i-1)}$, where the random walk $\sum_{i=1}^{n}Z_i, n\in\mathbb{N},$ is generated by $N$ periodically occurring distributions and the integer-valued and non-negative random variables $Z_1,\,Z_2,\,\ldots$ are independent. The considered random walk generates so-called multi seasonal discrete time risk model, and a known distribution of random variable $\mathcal{M}$ enables to calculate ultimate time ruin or survival probability. Verifying obtained theoretical statements we demonstrate several computational examples for survival probability $\mathbb{P}(\mathcal{M}< u)$ when $N=2,\,3$ or $10$.
\end{abstract}

\section{Introduction}\label{in}
The sequence of sums of independent random variables $\sum_{i=1}^{n}Z_i$,$\, n\in\mathbb{N}$, is usually called the {\it random walk}. Originally, random walk was introduced in \cite{Pearson}. There are dozens of versions of random walks appearing in many research areas and containing various interesting properties, which attract scientific attention, see for instance \cite{Erdos}, \cite{TBK} and references therein. In insurance, random walk appears arguing that the insurers surplus $W_{u,\kappa}$ for time moments $n\in\mathbb{N}_0:=\mathbb{N}\cup\{0\}$ might be expressed by the equation
$$
W_{u,\kappa}(n):=u+\kappa n-\sum_{i=1}^{n}Z_i,\,W_{u,\kappa}(0):=u,
$$
 where $u\in\mathbb{N}_0$ is interpreted as initial surplus, $\kappa\in\mathbb{N}$ denotes an income rate per unit of time (premium), and random variables $Z_i$ represent independent  random expenses. Therefore, playing a game of gain and loose, we are interested in a likelihood that $W_{u,\kappa}(n)>0$ for all $n\in \mathbb{N}$.

 Such type of models as the given one $W_{u,\kappa}(n)$ are called the {\it discrete time risk models} and there are many versions of them originating from the more general E. Sparre  Andersen's collective risk model introduced in \cite{Andersen}. Classical and recent works on the subject are \cite{dv-1997,dv-1999, dg-1988, di-1999, dw-1991, Gerber, Gerber1, lp-2006, ll-2008, pl-1997, Shiu, yz-2006}.
 
If $\kappa=1$ and random generators $Z_1,\,Z_2,\,\ldots$ are integer valued, independent and identically distributed, then the insurer's surplus $W_{u,1}$ is called the homogeneous discrete time risk model. Such model admits the simplest representation. Its finite and ultimate time ruin probabilities are:
\begin{equation}
 \begin{aligned}
\psi(u,t)&:=\mathbb{P}\left(\bigcup_{n=1}^{t}\Big\{W_{u,1}(n)\leqslant 0\Big\}\right),\ t\in\mathbb{N},\\
\psi(u)&:=\mathbb{P}\left(\bigcup_{n=1}^{\infty}\Big\{W_{u,1}(n)\leqslant 0\Big\}\right).
\end{aligned}
\end{equation}
Then, the values of the finite time ruin probability of the homogeneous discrete time risk model can be expressed by 
\begin{equation}\label{fi}
\begin{aligned}
\psi(u,1)&=\mathbb{P}\left(Z_1 > u\right),\,  \psi(u,t)=\\ &=\psi(u,1)+\sum_{k=1}^{u}\psi(u+1-k,t-1)\mathbb{P}\left(Z_1=k\right),\\
\end{aligned}
\end{equation}
for $u\in \mathbb{N}_0,\,t\in\mathbb{N}$, see for example \cite{d-2005,dg-1988,dw-1991,s-1989}, while the ultimate time ruin probability, under assumption $\mathbb{E}Z_1<1$, satisfies
\begin{equation}\label{fii}
 \begin{aligned}
 \psi(0)&=\mathbb{E}Z_1,\\ \psi(u)&=\sum_{j=1}^{u-1}\mathbb{P}\left(Z_1>j\right)\psi(u-j)+\sum_{j=u}^{\infty}\mathbb{P}\left(Z_1>j\right)), \quad u\in\mathbb{N},
 \end{aligned}
\end{equation}
 see \cite{d-2005}, also \cite{dw-1991, Shiu, s-1989} and references therein.
 
 If $\kappa>1$ or random generators $Z_1,\,Z_2,\,\ldots$ are not identically distributed, then the behaviour of process $W_{u,\kappa}$ becomes much more complicated. However, for the finite time ruin probability a recursive formula, similar to \eqref{fi}, holds when $\kappa=1$ and $Z_1,\,Z_2,\,\ldots$ is an arbitrary sequence of independent, non-negative and integer valued random variables, see \cite{BBS}. For the ultimate time ruin probability it is also possible to derive a recursive relation of type \eqref{fii}, however there arise difficulties in determining the values of function $\psi(u)$ for $u\in\{0,\,1,\,\ldots,\,d\}$, where $d$ is model dependent. In papers \cite{DS, GKS, GS, GS1} several multi seasonal discrete time risk models were examined. It was demonstrated there that initial values of the function $\psi$ can be found from limits of some special recurrences. In this paper, we consider a general $N$-seasonal discrete time risk model. We propose a new, generating functions based, method to find the initial values of the survival function $\varphi:=1-\psi$.

\section{Model description and preliminaries}\label{md}

In this section, we describe $N$-seasonal discrete time risk model in detail. We say that insurer's wealth $W_u(n)$ varies according to \textit{$N$-seasonal $(N\in\mathbb{N})$ discrete time risk model} (denoted by \textit{DTRM(N)}) if, for every moment of time $n\in\mathbb{N}_0$,
\begin{align}\label{pop}
W_u(n):=u+n-\sum_{j=1}^nZ_j,\,W_u(0):=u.
\end{align}
Here $u\in \mathbb{N}_0$ denotes initial insurer's wealth and $Z_1,\,Z_2,\,\ldots$ are independent, non-negative and integer-valued random variables such that $Z_j\mathop{=}\limits^{d}Z_{N+j}$ for each $j\in\mathbb{N}$. According to the definition, each \textit{DTRM(N)} is generated by initial wealth $u\in\mathbb{N}$ and collection of independent, non-negative and integer-valued random variables $Z_1,\,Z_2,\,\ldots,\, Z_N$.

    As it was mentioned in Section \ref{in}, we are interested to calculate \textit{the survival probability}
\begin{equation}\label{1}
\varphi(u):=\mathbb{P}\left(\bigcap_{n=1}^{\infty}\left\{W_u(n)>0\right\}\right)
=\mathbb{P}\left(\sup_{n\geqslant 1}\sum_{j=1}^n\left(Z_j-1\right)<u\right).
\end{equation}
Mathematically, the survival probabilities admit simpler representations compared to the ruin probabilities. 

In turn, the values of $\varphi(u)$ are dependent on the condition
\begin{equation}\label{2}
\mathbb{E}S_N<N,
\end{equation}
where $S_N:=\sum_{j=1}^NZ_j$, see Theorems \ref{t1} and \ref{t2} below. In insurance, the condition \eqref{2} is called the \textit{net profit condition} and it means that occurring claim amounts on average are smaller than collected premiums. 

    In view of \eqref{1}, we derive that
\begin{scriptsize}%
\begin{align}\label{2+}
\varphi(u)&=\mathbb{P}\left(\bigcap_{n=1}^{N}\Big\{W_u(n)>0\Big\}\cap
\bigcap_{n=N+1}^{\infty}\Big\{W_u(n)>0\Big\}\right)\nonumber\\
&=\mathbb{P}\left(\bigcap_{n=1}^{N}\bigg\{\sum_{j=1}^nZ_j\leqslant u+n-1\Bigg\}\cap
\bigcap_{n=N+1}^{\infty}\bigg\{u+N-\sum_{j=1}^NZ_j+n-N-\sum_{j=N+1}^nZ_j>0\Bigg\}\right)\nonumber\\
&=\sum_{\substack{i_1\leqslant u\\i_1+i_2\leqslant u+1\\i_1+i_2+i_3\leqslant u+2\\ \ldots \\
i_1+i_2+\ldots+i_N\leqslant u+N-1}}\hspace{-5mm}\mathbb{P}(Z_1=i_1)\mathbb{P}(Z_2=i_2)\cdots\mathbb{P}(Z_N=i_N)\,
\varphi\left(u+N-\sum_{j=1}^Ni_j\right).
\end{align}
\end{scriptsize}%
In order to find the values $\varphi(N),\,\varphi(N+1),\,\ldots$ from recurrence (\ref{2+}) we must know $\varphi(0),\,\varphi(1),\,\ldots,\, \varphi(N-1)$. Therefore, our focus is on finding these initial values of $\varphi$.
Let us note that the number of initial values required for the recurrence relation \eqref{2+} depends on the smallest values of random variables $Z_1,\,Z_2,\,\ldots,\,Z_N$. One may observe that the order or recurrence in \eqref{2+} reduces if some random variables $Z_1,\,Z_2,\,\ldots,\,Z_N$ do not attain ''small'' values, i.e. $\mathbb{P}(Z_n> k)=1$ for some $n=1,\,2,\,\ldots,\,N$ and $k\geqslant0$. For instance, if $Z_N\geqslant N-1$, then it follows from \eqref{2+} that
$$
\varphi(0)=\mathbb{P}\left(\sum_{j=1}^NZ_j=N-1\right)\,\varphi(1)
$$
and there is only $\varphi(0)$ we need to know to use the formula \eqref{2+}. More over, we observe that there are $\binom{2N-1}{N}$ integer and non-negative solutions $(i_1,\,i_2,\,\ldots,\,i_N)$ satisfying the inequality $i_1+i_2+\ldots+i_N<N$. In our context, this means that there are $\binom{2N-1}{N}$ options where the values of random variables $Z_1,\,Z_2,\,\ldots,\,Z_N$ can start and not to violate the net profit condition. 
Such type of combinatorial questions appear in many probabilistic and combinatorial books, see for example \cite{Ross}.

Before formulating the main result, we introduce several notations. We denote the following random variables which are related to shifted random walks appearing in formula \eqref{1}:
\begin{equation}\label{3}
\begin{aligned}
\mathcal{M} :=\mathcal{M}_1 :=\sup_{k\geqslant 1}&\left(\sum_{j=1}^k(Z_j-1)\right)^+,\,
\mathcal{M}_2:=\sup_{k\geqslant 2}\left(\sum_{j=2}^k(Z_j-1)\right)^+,\, \ldots,\,\\
&\mathcal{M}_N:=\sup_{k\geqslant N}\left(\sum_{j=N}^k(Z_j-1)\right)^+,
\end{aligned}
\end{equation}

where $a^+=\max\{a,0\}$ denotes the positive part of $a\in\mathbb{R}$. Let us observe that in general, the random variables $\mathcal{M}_1,\,\mathcal{M}_2,\,\ldots,\,\mathcal{M}_N$ are extended, i.e. they may attain the infinity with positive probability. However, if the net profit condition is satisfied, then these random variables are not extended, i.e.
$$
\lim_{u\rightarrow\infty}\mathbb{P}\left(\mathcal{M}_k< u\right)=1
$$
for every $k\in\{1,\,2,\,\ldots,\,N\}$. See Lemma \ref{lema1} in Section \ref{al} for the proof.\\
We now denote the local probabilities
\begin{equation}\label{0+}
m_j^{(k)}:=\mathbb{P}\left(\mathcal{M}_k=j\right),\ \  z_j^{(k)}:=\mathbb{P}\left(Z_k=j\right),\ \ s_j^{(N)}:=\mathbb{P}\left(S_N=j\right),
\end{equation}
where $k\in\{1,\,2,\,\ldots,\,N\}$ and $j\in\mathbb{N}_0$.
The defined probabilities and formula \eqref{1} imply 
\begin{equation}\label{cai}
\varphi(n+1)=\mathbb{P}(\mathcal{M}\leqslant n)=\sum_{j=0}^nm_j^{(1)}
\end{equation}
for all $n\in\mathbb{N}_0$.

\section{Main results}\label{mr}

In this section, we formulate two key theorems which are used to set up an algorithms for the survival probability $\varphi$ calculation of the \textit{DTRM(N)}. Before formulating these main results, we recall that the generating function of an integer-valued and non-negative random variable $Z$ is the complex function
$$
G_Z(s):=\sum_{k=0}^\infty\mathbb{P}(Z=k)s^k,
$$
where $s\in\mathbb{C}$ and $|s|\leqslant 1$. It is easy to check that if $Z_1$ and $Z_2$ are two independent random variables, then
$$
G_{Z_1+Z_2}(s)=G_{Z_1}(s)G_{Z_2}(s).
$$
Hence, for the \textit{DTRM(N)} we have that
$$
G_{S_N}(s)=\prod_{k=1}^NG_{Z_k}(s).
$$
We use notations $G_Z(s)$ and $G(s)$ interchangeably meaning the generating function of some non-negative and integer valued r.v.
\begin{theorem}\label{t1}
Suppose that the $N$-seasonal discrete time risk model
is generated by random variables $Z_1,\,Z_2,\,\ldots,\, Z_N$, with distributions described by probabilities \eqref{0+}. If the net profit condition \eqref{2} is satisfied, then the following five statements hold:\\
\indent{\rm (i)} The probability generating functions of $Z_1,\,Z_2,\,\ldots,\,Z_N$ and $\mathcal{M}_N$, for $0<|s|\leqslant1$, are related by
\begin{multline}\label{five}
(s-1)\left(m_0^{(1)}z_0^{(N)}+\sum_{k=1}^{N-1}m_0^{(k+1)}z_0^{(k)}G_{Z_N}(s)
\,s^{-k}\prod_{j=1}^{k-1}G_{Z_j}(s)\right)=\\=s\,G_{\mathcal{M}_N}(s)\left(1-\frac{G_{S_N}(s)}{s^N}\right).
\end{multline}
\indent{\rm (ii)}
The probabilities $m_0^{(1)},\,m_0^{(2)},\,\ldots,\,m_0^{(N)}$ and $z_0^{(1)},\,z_0^{(2)},\,\ldots,\,z_0^{(N)}$ satisfy the equation
\begin{equation}\label{aiai+++}
m_0^{(1)}z_0^{(N)}+\sum_{k=1}^{N-1}m_0^{(k+1)}z_0^{(k)}=N-\mathbb{E}S_N.
\end{equation}
\indent{\rm (iii)} If $\alpha, 0<|\alpha|< 1,$ is a root of $G_{S_N}(s)=s^N$, then
\begin{equation}\label{aiai+}
m_0^{(1)}z_0^{(N)}+\sum_{k=1}^{N-1}m_0^{(k+1)}z_0^{(k)}G_{Z_N}(\alpha)\,\alpha^{-k}\prod_{j=1}^{k-1}G_{Z_j}(\alpha)=0.
\end{equation}
\indent{\rm (iv)} If $\alpha, 0<|\alpha|< 1$, is a root of $G_{S_N}(s)=s^N$ with multiplicity $\varkappa, \varkappa\in\{2,\,3,\,\ldots,\,N-1\},$ then 
\begin{equation}\label{aiai++}
\begin{aligned}
&\sum_{k=1}^{N-1}m_0^{(k+1)}z_0^{(k)}\frac{d^l}{d s^l}\left(G_{Z_N}(s)\,s^{-k}
\prod_{j=1}^{k-1}G_{Z_j}(s)\right)\Bigg{|}_{s=\alpha}=0,\\ &l\in\{1,\,2,\,\ldots,\,\varkappa-1\},
\end{aligned}
\end{equation}
where $d^l/ds^l(\cdot)|_{s=\alpha}$ denotes the $l$'th derivative at $s=\alpha$.

\indent{\rm (v)} The probabilities $m_n^{(1)},\,m_n^{(2)},\,\ldots,\,m_n^{(N)}$ and $z_n^{(1)},\,z_n^{(2)},\,\ldots,\,z_n^{(N)}$ for $n\in\mathbb{N}$ satisfy the system of equations
\begin{align}\label{aiai++++}
\begin{cases}
z_0^{(N)}m_n^{(1)}&=\,m_{n-1}^{(N)}-\sum\limits_{j=0}^{n-1}z_{n-j}^{(N)}m_j^{(1)}-m_0^{(1)}z_0^{(N)}\ind_{\{n=1\}}\\
z_0^{(1)}m_n^{(2)}&=\,m_{n-1}^{(1)}-\sum\limits_{j=0}^{n-1}z_{n-j}^{(1)}m_j^{(2)}-m_0^{(2)}z_0^{(1)}\ind_{\{n=1\}}\\
z_0^{(2)}m_n^{(3)}&=\,m_{n-1}^{(2)}-\sum\limits_{j=0}^{n-1}z_{n-j}^{(2)}m_j^{(3)}-m_0^{(3)}z_0^{(2)}\ind_{\{n=1\}}\\
&\,\,\vdots\\
z_0^{(N-1)}m_n^{(N)}&=\,m_{n-1}^{(N-1)}-\sum\limits_{j=0}^{n-1}z_{n-j}^{(N-1)}m_j^{(N)}-m_0^{(N)}z_0^{(N-1)}\ind_{\{n=1\}}
\end{cases}.
\end{align}
\end{theorem}

\begin{theorem}\label{t2}
Suppose that the $N$-seasonal discrete time risk model is generated by random variables $Z_1,\,Z_2,\,\ldots,\, Z_N$ and the net profit condition \eqref{2} is unsatisfied:\\
\indent {\rm (i)} If  $\mathbb{E} S_N > N$, then $\varphi(u)=0$ for all $u\in\mathbb{N}_0$.\\
\indent {\rm (ii)} If $\mathbb{E} S_N = N$ and  $\mathbb{P}(S_N=N)<1$, then $\varphi(u)=0$ for all $u\in\mathbb{N}_0$.\\
\indent {\rm (iii)} If $\mathbb{P}(S_N=N)=1$, then random variables $Z_1,\,Z_2,\,\ldots,\,Z_N$ are degenerate and
$u+t^*-\sum_{k=1}^{t^*}Z_k\leqslant 0$ implies $\varphi(u)=0$,
while $u+t^*-\sum_{k=1}^{t^*}Z_k\geqslant 1$ implies $\varphi(u)=1$, where $t=t^*\in\{1,2,\ldots,N\}$ minimizes the difference
$t-\sum_{k=1}^{t}Z_k$.
\end{theorem}

It is worth mentioning that connection between characteristics of collective risk models and roots of a certain generating functions is not surprising. The simple example could be Lundberg inequality (see \cite{AR} and \cite{Lundberg}), which is based on a root of a certain exponential moment. One may point to \cite{Rin_San} too, where the homogeneous model $DTRM(1)$ is studied with particular claims distribution. See also the source \cite{Kendall}, which includes a nice historical overview on branching and Galton-Watson processes.

\section{Bi-seasonal discrete time risk  model}\label{bi}

Originally the ultimate time ruin probability $\psi(u)=1-\varphi(u)$ for the  bi-risk  discrete time risk model \textit{DTRM(2)} was studied in \cite{DS}. In this section, we present new results for this model that simplifies the survival probability calculation when $N=2$. More precisely, the following assertion holds to get the initial values $\varphi(0)$ and $\varphi(1)$ needed for the recurrence relation \eqref{2+}.
\begin{theorem}\label{t3}
Suppose that the bi-risk discrete time risk model (DTRM(2)) is generated by random variables $Z_1,\,Z_2$ and the net profit condition holds $\mathbb{E}S_2=\mathbb{E}\big(Z_1+Z_2\big)<2$. Then, the following three statements are correct.\\
\indent {\rm (i)} If $z_0^{(1)}=\mathbb{P}(Z_1=0)>0$ and $z_0^{(2)}=\mathbb{P}(Z_2=0)>0$, then:
\begin{align*}
\varphi(0)=\left(2-\mathbb{E}S_2\right)\frac{\alpha}{\alpha-G_{Z_2}(\alpha)},\ \
\varphi(1)=\frac{2-\mathbb{E}S_2}{z_0^{(2)}}\frac{G_{Z_2}(\alpha)}{G_{Z_2}(\alpha)-\alpha},
\end{align*}
where $\alpha\in(-1,0)$, is the unique root of $G_{S_2}(s)=s^2$ and $G_{Z_2}(s)$ is the generating function of r.v. $Z_2$.\\
\indent {\rm (ii)} If $z_0^{(1)}>0$ and $z_0^{(2)}=0$, then:
\begin{align*}
\varphi(0)=2-\mathbb{E}S_2,\ \
\varphi(1)=\frac{2-\mathbb{E}S_2}{z_0^{(1)}z_1^{(2)}}.
\end{align*}
\indent {\rm (iii)} If $z_0^{(1)}=0$ and $z_0^{(2)}>0$, then:
\begin{align*}
\varphi(0)=0,\ \
\varphi(1)=\frac{2-\mathbb{E}S_2}{z_0^{(2)}}.
\end{align*}
In addition, $z_0^{(1)}=z_0^{(2)}=0$ causes $\mathbb{E}S_2\geqslant2$
and survival is impossible, i.e. $\varphi(u)=0$ for all $u\in\mathbb{N}_0$, except few trivial cases when $\mathbb{P}(S_2=2)=1$.
\end{theorem}

\section{Auxiliary lemmas}\label{al}

In this section we formulate several lemmas which are used to prove Theorem \ref{t1}. The first lemma shows that under the net profit condition \eqref{2} random variables $\mathcal{M}_1,\,\mathcal{M}_2,\, \ldots,\, \mathcal{M}_N$ are not extended, i.e. $\mathbb{P}\left(\mathcal{M}_k<\infty\right)=1$ for all $k\in\{1,\,2,\,\ldots,\,N\}$.

\begin{lemma}\label{lema1}
Suppose that the DTRM(N) is generated by random variables $Z_1,\,Z_2,\,\ldots,\,Z_N$ and the net profit condition \eqref{2} is satisfied. Then, for all $k\in\{1,\,2,\,\ldots,\,N\}$,
$$
\lim_{u\rightarrow\infty}\mathbb{P}\left(\mathcal{M}_k< u\right)=1.
$$
\end{lemma}
\begin{proof}
We consider the case $k=1$ only and note that other cases $k\in\{2,\,3,\,\ldots,\,N\}$ can be derived similarly.
According to the strong law of large numbers for random variables $\widehat{Z}_j:=Z_j-1,\, j\in\mathbb{N},$ we obtain that the sequence
$$
\frac{1}{n}\sum_{j=1}^n\widehat{Z}_j=\frac{1}{N}\left(\frac{N}{n}\sum_{\substack{{j=1}\\ j\equiv 1\,{\rm mod}\,N}}^n\widehat{Z}_j+\ldots+\frac{N}{n}\sum_{\substack{{j=1}\\ j\equiv N\,{\rm mod}\,N}}^n\widehat{Z}_j\right),
$$
as $n\to\infty$, converges almost surely to the negative number $(\mathbb{E}S_N-N)/N:=-\mu$.  Therefore,
$$
\mathbb{P}\left(\sup_{k\geqslant n}\bigg{|}\frac{1}{k}\sum_{j=1}^k\widehat{Z}_j+\mu\bigg{|}<\frac{\mu}{2}\right)
\mathop{\rightarrow}_{n\rightarrow\infty}1.
$$
Therefore, for $0<\varepsilon<1$, the previous relation yields
\begin{align*}
\mathbb{P}\left(\bigcap_{k=n}^\infty\bigg{\{}\sum_{j=1}^kZ_j\leqslant 0\bigg{\}}\right)\geqslant
\mathbb{P}\left(\sup_{k\geqslant n}\bigg{|}\frac{1}{k}\sum_{j=1}^k\widehat{Z}_j+\mu\bigg{|}<\frac{\mu}{2}\right)\geqslant 1-\varepsilon
\end{align*}
if $n>M=M(\varepsilon)$.\\ Consequently, for such  $\varepsilon$, we have that
\begin{align*}
\mathbb{P}(\mathcal{M}_1< u)&=\mathbb{P}\left(\sup_{k\geqslant 1}\left(\sum_{j=1}^k\widehat{Z}_j\right)^+< u\right)\\ &\geqslant
\mathbb{P}\left(\bigcap_{k=1}^{M-1}\left\{\sum_{j=1}^k\widehat{Z}_j< u\right\}\cap\bigcap_{k=M}^{\infty}\left\{\sum_{j=1}^k\widehat{Z}_j\leqslant 0\right\}\right)
\\
&\geqslant\mathbb{P}\left(\bigcap_{k=1}^{M-1}\left\{\sum_{j=1}^k\widehat{Z}_j< u\right\}\right)+
\mathbb{P}\left(\bigcap_{k=M}^{\infty}\left\{\sum_{j=1}^k\widehat{Z}_j\leqslant 0\right\}\right)-1\\
&\geqslant\mathbb{P}\left(\bigcap_{k=1}^{M-1}\left\{\sum_{j=1}^k\widehat{Z}_j< u\right\}\right)-\varepsilon.
\end{align*}
The last estimate implies
$$
\liminf_{u\rightarrow\infty}\mathbb{P}\left(\mathcal{M}_1< u\right)\geqslant 1-\varepsilon
$$
and the assertion of Lemma follows due to $\varepsilon>0$ being as small as we want.
\end{proof}

Let us denote
$$
\varphi_k(u):=\mathbb{P}\left(\sup_{n\geqslant k}\sum_{j=k}^n\,\left(Z_j-1\right)<u\right),\ k\in\{1,\,2,\,\ldots,\,N\}.
$$
Obviously, according to \eqref{1}, $\varphi_1(u)=\varphi(u)$ for $u\in\mathbb{N}_0$. Lemma \ref{lema1} implies the following property of the defined probabilities $\varphi_k(u)$.

\begin{corollary}\label{papa-}
If the DTRM(N) is generated by random variables $Z_1,\,Z_2,\,\ldots,\,Z_N$ and the net profit condition \eqref{2} is satisfied, then $\lim_{u\rightarrow\infty}\varphi_k(u)=1$ for all $k\in\{1,\,2,\,\ldots,\,N\}$.
\end{corollary}

The next lemma provides one important distribution property of random variables $\mathcal{M}_1, \mathcal{M}_2,\ldots, \mathcal{M}_N$ which are defined in \eqref{3}.

\begin{lemma}\label{lema2}
Suppose that the DTRM(N) is generated by random variables $Z_1,\,Z_2,\,\ldots,\,Z_N$ with $N\geqslant 2$. If the net profit condition \eqref{2} is satisfied, then:
$$
\left(\mathcal{M}_k+\widehat{Z}_{k-1}\right)^+\mathop{=}\limits^{d}\mathcal{M}_{k-1},\ k\in\{2,\,3,\,\ldots,\,N\},\ \left(\mathcal{M}_1+\widehat{Z}_{N}\right)^+\mathop{=}\limits^{d}\mathcal{M}_{N}.
$$
\end{lemma}

\begin{proof}
We prove the last equality only because the other ones can be derived by the same arguments. According to Lemma \ref{lema1}, the random variable $\mathcal{M}_1$ is not extended - it attains the finite values only. Hence,
\begin{footnotesize}%
\begin{align*}
&\left(\mathcal{M}_1+\widehat{Z}_N\right)^+
\mathop{=}\limits^{d}\left(\max_{k\geqslant 1}\left\{\left(\sum_{j=1}^k\widehat{Z}_j\right)^+\right\}+\widehat{Z}_N\right)^+\\
&=\left(\max\left\{\widehat{Z}^+_1,\, (\widehat{Z}_1+\widehat{Z}_2)^+,\, (\widehat{Z}_1+\widehat{Z}_2+\widehat{Z}_3)^+,\,\ldots\right\}+\widehat{Z}_N\right)^+\\
&=\left(\max\left\{0,\, \widehat{Z}_1,\, \widehat{Z}_1+\widehat{Z}_2,\, \widehat{Z}_1+\widehat{Z}_2+\widehat{Z}_3,\,\ldots\right\}+\widehat{Z}_N\right)^+\\
&=\left(\max\left\{\widehat{Z}_N,\, \widehat{Z}_N+\widehat{Z}_1,\, \widehat{Z}_N+\widehat{Z}_1+\widehat{Z}_2,\, \widehat{Z}_N+ \widehat{Z}_1+\widehat{Z}_2+\widehat{Z}_3,\,\ldots\right\}\right)^+\\
&\mathop{=}\limits^{d}\left(\max\left\{\widehat{Z}_N,\, \widehat{Z}_N+\widehat{Z}_{N+1},\, \widehat{Z}_N+\widehat{Z}_{N+1}+\widehat{Z}_{N+2},\, \widehat{Z}_N+ \widehat{Z}_{N+1}+\widehat{Z}_{N+2}+\widehat{Z}_{N+3},\,\ldots\right\}\right)^+\\
&\mathop{=}\limits^{d}\left(\max_{k\geqslant N}\left\{\sum_{j=N}^k\widehat{Z}_j\right\}\right)^+
\mathop{=}\limits^{d}\max_{k\geqslant N}\left\{\left(\sum_{j=N}^k\widehat{Z}_j\right)^+\right\}
\mathop{=}\limits^{d}\mathcal{M}_N,
\end{align*}
\end{footnotesize}%
and assertion of the lemma follows.
\end{proof}

The next lemma provides relations between the generating functions $G_{Z_k}(s)$ and $G_{\mathcal{M}_k}(s)$ for every $k\in\{1,2,\ldots,N\}$. Here
$$
G_{\mathcal{M}_k}(s)=\sum_{j=0}^\infty\mathbb{P}\left(\mathcal{M}_k=j\right)s^j,\,|s|\leqslant1.
$$

\begin{lemma}\label{lema3}
Suppose the DTRM(N) is generated by random variables $Z_1,\,Z_2,\,\ldots,\,Z_N$, $N\geqslant 2$. If the net profit condition \eqref{2} is satisfied, then, for every $s, s\neq 0, |s|\leqslant 1,$ the following equalities hold
\begin{align}\label{OR}
\begin{cases}
(s-1)\,m_0^{(2)}z_0^{(1)}&=s\,G_{\mathcal{M}_1}(s)-G_{Z_1}(s)\,G_{\mathcal{M}_2}(s)\\
(s-1)\,m_0^{(3)}z_0^{(2)}&=s\,G_{\mathcal{M}_2}(s)-G_{Z_2}(s)\,G_{\mathcal{M}_3}(s)\\
&\,\vdots\\
(s-1)\,m_0^{(N)}z_0^{(N-1)}&=s\,G_{\mathcal{M}_{N-1}}(s)-G_{Z_{N-1}}(s)\,G_{\mathcal{M}_N}(s)\\
(s-1)\,m_0^{(1)}z_0^{(N)}&=s\,G_{\mathcal{M}_{N}}(s)-G_{Z_{N}}(s)\,G_{\mathcal{M}_1}(s)
\end{cases}.
\end{align}
\end{lemma}

\begin{proof} We prove the first equation in (\ref{OR}) only and note that other equations can be proved by the same arguments. According to Lemma \ref{lema2}, we have
$$
\Big(\mathcal{M}_2+\widehat{Z}_1\Big)^+\mathop{=}\limits^{d}\mathcal{M}_1.
$$

Therefore, using a basic expectation properties, for $s, s\neq 0, |s|\leqslant 1,$ we get

\begin{align*}
G_{\mathcal{M}_1}(s)&=\mathbb{E}\left(s^{\mathcal{M}_1}\right)
=\mathbb{E}\Big(s^{\left(\mathcal{M}_2+\widehat{Z}_1\right)^+}\,\Big)
=\mathbb{E}\left(s^{\left(\mathcal{M}_2+\widehat{Z}_1\right)^+}\sum_{j=0}^\infty\ind_{\{\mathcal{M}_2=j\}}\right)\\
&=\mathbb{E}\left(\sum_{j=0}^\infty s^{\left(j+{Z}_1-1\right)^+}\ind_{\{\mathcal{M}_2=j\}}\right)=
\sum_{j=0}^\infty m_j^{(2)}\mathbb{E}\Big(s^{\left(j+{Z}_1-1\right)^+}\Big)\\
&=m_0^{(2)}\mathbb{E}\Big(s^{(Z_1-1)^+}\Big)+\sum_{j=1}^\infty m_j^{(2)}\mathbb{E}\Big(s^{\left(j+{Z}_1-1\right)}\Big)\\
&=m_0^{(2)}\bigg(z_0^{(1)}+\sum_{j=1}^\infty z_j^{(1)}s^{(j-1)}\bigg)+\frac{1}{s}\, \mathbb{E}\big(s^{Z_1}\big)
\sum_{j=1}^\infty m_j^{(2)}s^j\\
&=m_0^{(2)}\left(z_0^{(1)}+\frac{1}{s}\left(G_{Z_1}(s)-z_0^{(1)}\right)\right)+\frac{1}{s}G_{Z_1}(s)
\left(G_{\mathcal{M}_2}(s)-m_0^{(2)}\right)\\
&=m_0^{(2)}z_0^{(1)}\left(1-\frac{1}{s}\,\right)+ \frac{1}{s}\,G_{Z_1}(s)\,G_{\mathcal{M}_2}(s).
\end{align*}
Since $s\neq 0$, the desired result follows by multiplying both sides of the last equality by $s$. Lemma is proved.
\end{proof}

The following lemma investigates when the expectations of  random variables $\mathcal{M}_1,\,\mathcal{M}_2,\,\ldots,\,\mathcal{M}_N$ are finite.

\begin{lemma}\label{lema5}
Suppose the DTRM(N) is generated by random variables $Z_1,\,Z_2,\,\ldots,\,Z_N$, $N\geqslant 2$ and recall that random variables $\mathcal{M}_1,\,\mathcal{M}_2,\,\ldots,\,\mathcal{M}_N$ are defined in \eqref{3}. If the net profit condition \eqref{2} is satisfied, then:
\\ \indent {\rm (i)} $\ \mathbb{E}\mathcal{M}_1<\infty$ \ $\Leftrightarrow$\  $\mathbb{E}\mathcal{M}_k<\infty$  for all $k\in\{1,\,2,\,\ldots,\,N\}$,\\
\indent {\rm (ii)} $\mathbb{E}\mathcal{M}_1=\infty$ \ $\Leftrightarrow$\  $\mathbb{E}\mathcal{M}_k=\infty$  for all $k\in\{1,\,2,\,\ldots,\,N\}$.
\end{lemma}

\begin{proof}
According to Lemma \ref{lema1}, the random variables $\mathcal{M}_1,\,\mathcal{M}_2,\,\ldots,\,\mathcal{M}_N$ are non-extended - they attain the finite values only. Hence,
\begin{multline*}
\mathcal{M}_2+\widehat{Z}_1
\mathop{=}\limits^{d}\max_{k\geqslant 2}\left(\sum_{j=2}^k\widehat{Z}_j\right)^++\widehat{Z}_1
\mathop{=}\limits^{d}\widehat{Z}_1+\max\left\{0,\, \widehat{Z}_2,\,\widehat{Z}_2+\widehat{Z}_3,\,\ldots\right\}\\
\leqslant \max\left\{0,\widehat{Z}_1,\, \widehat{Z}_1+\widehat{Z}_2,\,\widehat{Z}_1+\widehat{Z}_2+\widehat{Z}_3,\,\ldots\right\}
\mathop{=}\limits^{d}\mathcal{M}_1,
\end{multline*}
with probability one.

Analogously, $\mathbb{P}\left(\mathcal{M}_k+\widehat{Z}_{k-1}\leqslant \mathcal{M}_{k-1}\right)=1$ for all $k\in\{3,\,4,\,\ldots,\,N\}$, and, consequently
\begin{equation}\label{pyp1}
\mathbb{E}\mathcal{M}_{k}+\mathbb{E}\widehat{Z}_{k-1}\leqslant\mathbb{E}\mathcal{M}_1,\ k\in\{1,\,2,\,\ldots,\,N\}.
\end{equation}
Similarly,
\begin{multline*}
\mathcal{M}_1
\mathop{=}\limits^{d}\max\left\{0,\,\widehat{Z}_1,\, \widehat{Z}_1+\widehat{Z}_2,\,\widehat{Z}_1+\widehat{Z}_2+\widehat{Z}_3,\,\ldots\right\}\\ \leqslant\max\left\{0,\,\widehat{Z}_1\right\}+\max\left\{0,\,\widehat{Z}_2,\,\widehat{Z}_2+\widehat{Z}_3,\,\ldots\right\}
\mathop{=}\limits^{d}\widehat{Z}_1^++\mathcal{M}_2,
\end{multline*}
with probability one, implying that,
$$
\mathbb{E}\mathcal{M}_1\leqslant\mathbb{E}\widehat{Z}_1^++\mathbb{E}\mathcal{M}_2,
$$
where $\mathbb{E}\widehat{Z}_1+z^{(1)}_0=\mathbb{E}\widehat{Z}_1^+<N-1+z^{(1)}_0$.

    By the same arguments
$$
\mathbb{E}\mathcal{M}_2\leqslant\mathbb{E}\widehat{Z}_2^++\mathbb{E}\mathcal{M}_3,
$$
where $\mathbb{E}\widehat{Z}_2+z^{(2)}_0=\mathbb{E}\widehat{Z}_2^+<N-1+z^{(2)}_0$,
and
\begin{equation}\label{pyp2}
\mathbb{E}\mathcal{M}_k\leqslant\mathbb{E}\widehat{Z}_k^++\mathbb{E}\mathcal{M}_{k+1},\ k\in\{1,\,2,\,\ldots,\,N-1\}
\end{equation}
with finite expectations $\mathbb{E}\widehat{Z}_k^+$.
The assertion of the lemma follows from the derived estimates \eqref{pyp1} and \eqref{pyp2}.
\end{proof}

In the next lemma we calculate one specific limit consisting from the first two infinite moments of integer valued random variable.

\begin{lemma}\label{lema4}
Let $X$ be an integer valued random variable and denote its probability generating function
$$
G_X(s)=\sum_{j=1}^\infty\mathbb{P}(X=j)s^j=:\sum_{j=1}^\infty p_js^j.
$$
If $\mathbb{E}X=\infty$, then
$$
\lim_{s\rightarrow 1^-}\frac{\big(G_X^{\prime}(s)\big)^2}{G_X^{\prime\prime}(s)}=0.
$$
\end{lemma}

\begin{proof} Suppose that $|s|<$. For any fixed natural $M\geqslant 2$ we get
\begin{equation*}
\frac{\big(G_X^{\prime}(s)\big)^2}{G_X^{\prime\prime}(s)}=
\frac{\Big(\sum_{j=1}^\infty jp_js^{j}\Big)^2}{\sum_{j=1}^\infty j(j-1)p_js^{j}}
=\frac{\Big(\sum_{j=M}^\infty jp_js^{j}\Big)^2}{\sum_{j=M}^\infty j(j-1)p_js^{j}}\,
\frac{\Bigg(1+\frac{\sum_{j=1}^{M-1}jp_js^{j}}{\sum_{j=M}^\infty jp_js^{j}}\Bigg)^2}
{\Bigg(1+\frac{\sum_{j=1}^{M-1}j(j-1)p_js^{j}}{\sum_{j=M}^\infty j(j-1)p_js^{j}}\Bigg)}.
\end{equation*}
Using a special version of Cauchy-Schwarz inequality (know as Sedrakyan's inequality, Bergström's inequality, Engel's form or Titu's lemma, see \cite{Caushcy}) we obtain that
\begin{equation*}
\frac{\Big(\sum_{j=M}^\infty jp_js^{j}\Big)^2}{\sum_{j=M}^\infty j(j-1)p_js^{j}}\leqslant
\sum_{j=M}^\infty\frac{j}{j-1}p_js^j\leqslant\frac{M}{M-1}\sum_{j=M}^\infty p_j.
\end{equation*}
Therefore,
\begin{multline*}
\limsup_{s\rightarrow 1^-}\frac{\big(G_X^{\prime}(s)\big)^2}{G_X^{\prime\prime}(s)}\leqslant
\frac{M}{M-1}\sum_{j=M}^\infty p_j\limsup_{s\rightarrow 1^-}\frac{\Bigg(1+\frac{\sum_{j=1}^{M-1} jp_js^{j}}{\sum_{j=M}^\infty jp_js^{j}}\Bigg)^2}
{\Bigg(1+\frac{\sum_{j=1}^{M-1} j(j-1)p_js^{j}}{\sum_{j=M}^\infty j(j-1)p_js^{j}}\Bigg)}=\\=\frac{M}{M-1}\sum_{j=M}^\infty p_j.
\end{multline*}
The assertion of the lemma follows because the last estimate is valid for any arbitrary $M\geqslant 2$.
\end{proof}

\section{Proof of Theorem \ref{t1}}\label{pt1}

In this section we prove the first main theorem.

\begin{proof}[Proof of Theorem 3.1]
We first prove the equality \eqref{aiai+}. The equations \eqref{OR} in Lemma \ref{lema3} 
\begin{tiny}%
\begin{align*}
\begin{cases}
\frac{s-1}{s}\,m_0^{(2)}z_0^{(1)}G_{Z_N}(s)&=G_{\mathcal{M}_1}(s)G_{Z_N}(s)-
\frac{1}{s}\,G_{\mathcal{M}_2}(s)G_{Z_N}(s)G_{Z_1}(s)\\
\frac{s-1}{s^2}\,m_0^{(3)}z_0^{(2)}G_{Z_N}(s)G_{Z_1}(s)&=\frac{1}{s}\,G_{\mathcal{M}_2}(s)G_{Z_N}(s)G_{Z_1}(s)-
\frac{1}{s^2}\,G_{\mathcal{M}_3}(s)G_{Z_N}(s)G_{Z_1}(s)G_{Z_2}(s)\\
&\,\,\vdots\\
\frac{s-1}{s^{N-1}}\,m_0^{(N)}z_0^{(N-1)}G_{Z_N}(s)G_{Z_1}(s)\cdots G_{Z_{N-2}}(s)&=\frac{1}{s^{N-2}}\,G_{\mathcal{M}_N}(s)G_{Z_N}(s)G_{Z_1}(s)\cdots G_{Z_{N-2}}(s)\\&\quad-
\frac{1}{s^{N-1}}\,G_{\mathcal{M}_N}(s)G_{Z_N}(s)G_{Z_1}(s)\cdots G_{Z_{N-1}}(s)\\
(s-1)\,m_0^{(1)}z_0^{(N)}&= s\,G_{\mathcal{M}_N}(s)-G_{\mathcal{M}_1}(s)G_{Z_N}(s)
\end{cases}.
\end{align*}
\end{tiny}%
Adding all these equations, we get \eqref{five}.

The equation \eqref{aiai+} of Theorem \ref{t1} follows now from the last relation.

    We now consider the equation \eqref{aiai++}. If $0<|s|<1$, then, from the equality \eqref{five}, we get
\begin{multline*}
m_0^{(1)}z_0^{(N)}+\sum_{k=1}^{N-1}m_0^{(k+1)}z_0^{(k)}G_{Z_N}(s)
\,s^{-k}\prod_{j=1}^{k-1}G_{Z_j}(s)=\\=\frac{1}{s-1}s^{1-N}\,G_{\mathcal{M}_N}(s)\left(s^N-G_{S_N}(s)\right).
\end{multline*}
The desired relation \eqref{aiai++} follows from the last equation having in mind that
\begin{equation*}
\frac{d^l}{ds^l}\left(s^N-G_{S_N}(s)\right)\Big{|}_{s=\alpha}=0
\end{equation*}
for all $l\in\{1,\,2,\,\ldots,\, \varkappa-1\}$ if the root $\alpha$ of $s^N=G_{S_N}(s)$ is of multiplicity $\varkappa$.

    We now prove the first equation \eqref{aiai+++} of the theorem. Calculating derivatives for all equalities in system \eqref{OR}, we get that for all $s$, $0<|s|<1$
\begin{footnotesize}%
\begin{align}\label{OR1}
\begin{cases}
m_0^{(2)}z_0^{(1)}&=G_{\mathcal{M}_1}(s)+sG_{\mathcal{M}_1}^\prime(s)-G^\prime_{Z_1}(s)G_{\mathcal{M}_2}(s)
-G_{Z_1}(s)G_{\mathcal{M}_2}^\prime(s)\\
m_0^{(3)}z_0^{(2)}&=G_{\mathcal{M}_2}(s)+sG_{\mathcal{M}_2}^\prime(s)-G^\prime_{Z_2}(s)G_{\mathcal{M}_3}(s)
-G_{Z_2}(s)G_{\mathcal{M}_3}^\prime(s)\\
&\,\,\vdots\\
m_0^{(N)}z_0^{(N-1)}&=G_{\mathcal{M}_{N-1}}(s)+sG_{\mathcal{M}_{N-1}}^\prime(s)-G^\prime_{Z_{N-1}}(s)G_{\mathcal{M}_N}(s)
-G_{Z_{N-1}}(s)G_{\mathcal{M}_N}^\prime(s)\\
m_0^{(1)}z_0^{(N)}&=G_{\mathcal{M}_N}(s)+sG_{\mathcal{M}_N}^\prime(s)-G^\prime_{Z_N}(s)G_{\mathcal{M}_1}(s)
-G_{Z_N}(s)G_{\mathcal{M}_1}^\prime(s)
\end{cases}.
\end{align}
\end{footnotesize}%
It is obvious that there are two possible cases: $\mathbb{E}\mathcal{M}_1<\infty$ or $\mathbb{E}\mathcal{M}_1=\infty$. If $\mathbb{E}\mathcal{M}_1<\infty$, then $\mathbb{E}\mathcal{M}_k<\infty$ for all $k\in\{1,\,2,\,\ldots,\,N\}$ due to the statement of Lemma \ref{lema5}. By letting $s\rightarrow 1^-$ from the system \eqref{OR1} we obtain
\begin{align*}
\begin{cases}
m_0^{(2)}z_0^{(1)}&=1+\mathbb{E}\mathcal{M}_1-\mathbb{E}Z_1-\mathbb{E}\mathcal{M}_2\\
m_0^{(3)}z_0^{(2)}&=1+\mathbb{E}\mathcal{M}_2-\mathbb{E}Z_2-\mathbb{E}\mathcal{M}_3\\
&\,\,\vdots\\
m_0^{(N)}z_0^{(N-1)}&=1+\mathbb{E}\mathcal{M}_{N-1}-\mathbb{E}Z_{N-1}-\mathbb{E}\mathcal{M}_N,\\
m_0^{(1)}z_0^{(N)}&=1+\mathbb{E}\mathcal{M}_N-\mathbb{E}Z_N-\mathbb{E}\mathcal{M}_1
\end{cases}.
\end{align*}
We derive the equality \eqref{aiai+++} adding all equalities of the last system.

\smallskip

If $\mathbb{E}\mathcal{M}_1=\infty$, then $\mathbb{E}\mathcal{M}_k=\infty$ for all $k\in\{1,\,2,\,\ldots,\,N\}$ due to Lemma \ref{lema5}. If that is the case, then, by summing equalities of system \eqref{OR1} and letting $s\to1^-$, we get
\begin{multline}\label{papa}
m_0^{(1)}z_0^{(N)}+\sum_{k=1}^{N-1}m_0^{(k+1)}z_0^{(k)}=\\= \lim_{s\rightarrow 1^-}\Bigg\{\sum_{k=1}^NG_{\mathcal{M}_k}(s)-
\sum_{k=1}^{N-1}G_{Z_k}^\prime(s)G_{\mathcal{M}_{k+1}}(s)-G^\prime_{Z_N}(s)G_{\mathcal{M}_1}(s)\nonumber\\
\quad +\big(s-G_{Z_N}(s)\big)G^\prime_{\mathcal{M}_1}(s)+\sum_{k=2}^N\big(s-G_{Z_{k-1}}(s)\big)G^\prime_{\mathcal{M}_k}(s)\Bigg\}.
\end{multline}
In view of Lemma \ref{lema1},
$$
\lim_{s\rightarrow 1^-}\sum_{k=1}^NG_{\mathcal{M}_k}(s)=N.
$$
The net profit condition \eqref{2} together with Lemma \ref{lema1} imply that
$$
\lim_{s\rightarrow 1^-}\left\{\sum_{k=1}^{N-1}G_{Z_k}^\prime(s)G_{\mathcal{M}_{k+1}}(s)+G^\prime_{Z_N}(s)G_{\mathcal{M}_1}(s)\right\}
=\mathbb{E}S_N.
$$
By L'Hospital's rule and Lemma \ref{lema4}, we get
\begin{multline}
\lim_{s\rightarrow 1^-}\big(s-G_{Z_N}(s)\big)G^\prime_{\mathcal{M}_1}(s)
=\lim_{s\rightarrow 1^-}\frac{\big(s-G_{Z_N}(s)\big)^\prime}{\bigg(\frac{1}{G'_{\mathcal{M}_1}(s)}\bigg)^\prime}
=\\=\lim_{s\rightarrow 1^-}\,\frac{1-G^\prime_{Z_N}(s)}{-\frac{G^{\prime\prime}_{\mathcal{M}_1}(s)}{\big(G^\prime_{\mathcal{M}_1}(s)\big)^2}}=0.
\end{multline}
By similar arguments,
$$
\lim_{s\rightarrow 1^-}\big(s-G_{Z_{k-1}}(s)\big)G^\prime_{\mathcal{M}_k}(s)=0
$$
for all $k\in\{2,\,3,\,\ldots,\,N\}$. 
Substituting these limits into \eqref{papa} we get the equality \eqref{aiai+++}.

    It remains to derive the system of equalities \eqref{aiai++++}. Let us consider the system of equalities \eqref{OR} from Lemma \ref{lema3}. 
By expanding the power series in the right hand-sides of \eqref{OR} and collecting the coefficients
we obtain
\begin{align}\label{deriv_n}
z_0^{(N)}m_n^{(1)}=m_{n-1}^{(N)}-\sum\limits_{j=0}^{n-1}z_{n-j}^{(N)}m_j^{(1)}-m_0^{(1)}z_0^{(N)}\ind_{\{n=1\}}
\end{align}
for all natural numbers $n$.
The derived equality (\ref{deriv_n}) is the first equality in the system \eqref{aiai++++}. The other equalities of this system \eqref{aiai++++} can be derived from \eqref{OR} analogously.
  Theorem \ref{t1} is proved.
\end{proof}

\section{Proof of Theorem \ref{t2}}\label{pt2}

In Section \ref{alg}, we present an algorithm to calculate the survival probability when DTRM(N) satisfies the net profit condition \eqref{2}. Under the breach of the net profit condition, the survival probability can be obtained directly from Theorem \ref{t2}. We present the proof of this theorem below.

\begin{proof}[Proof of Theorem 3.2]  First, let us assume that $\mathbb{E} S_N>N$. In view of the equality \eqref{2+}, for all $u\in\mathbb{N}_0$ we get
\begin{align*}
\varphi(u)&=\hspace{-6mm}\sum_{\substack{i_1\leqslant u\\i_1+i_2\leqslant u+1\\i_1+i_2+i_3\leqslant u+2\\ \ldots  \\
i_1+i_2+\ldots+i_N\leqslant u+N-1}}\hspace{-5mm}z_{i_1}^{(1)}z_{i_2}^{(2)}\cdots z_{i_N}^{(N)}\,
\varphi\left(u+N-\sum_{j=1}^Ni_j\right)\\
&=\sum_{j=1}^{u+N}s_{u+N-j}^{(N)}\,\varphi(j)-\sum_{j=1}^{N-1}\mu_j(u)\,\varphi(j),
\end{align*}
where $\mu_j(u)\geqslant0$ represent missing terms in $s_{u+N-j}^{(N)}$.
Therefore, defining that $\mu_0(u):=s_{u+N}^{(N)}$, we get
\begin{align*}
\varphi(u)&=\sum_{j=0}^{u+N}s_{u+N-j}^{(N)}\,\varphi(j)-\sum_{j=0}^{N-1}\mu_j(u)\,\varphi(j).
\end{align*}
By summing the both sides of the last equality by $u$, which varies from $0$ to some sufficiently large natural $v$, we obtain
\begin{align*}
\sum_{u=0}^v\varphi(u)&=\sum_{u=0}^v\sum_{j=0}^{u+N}s_{u+N-j}^{(N)}\,\varphi(j)
-\sum_{u=0}^v\sum_{j=0}^{N-1}\mu_j(u)\,\varphi(j)\\
&=\sum_{j=0}^{N-1}\varphi(j)\sum_{u=0}^vs_{u+N-j}^{(N)}+\sum_{j=N}^{v+N}\varphi(j)\sum_{u=j-N}^{v}s_{u+N-j}^{(N)}
-\\&-\sum_{j=0}^{N-1}\varphi(j)\sum_{u=0}^v\mu_j(u).
\end{align*}
Or, equivalently,
\begin{align}\label{ga}
&\sum_{u=0}^{v+N}\varphi(u)\left(1-\sum_{j=0}^{N+v-u}s_{j}^{(N)}\right)-\sum_{u=v+1}^{v+N}\varphi(u)=\nonumber\\
&=\sum_{j=0}^{N-1}\varphi(j)\left(\sum_{u=0}^vs_{u+N-j}^{(N)}-\sum_{u=j-N}^{v}s_{u+N-j}^{(N)}
-\sum_{u=0}^{v}\mu_j(u)\right),\nonumber\\
&\sum_{u=v+1}^{v+N}\varphi(u)-\sum_{u=0}^{v+N}\varphi(u)\mathbb{P}\left(S_N\geqslant v+N-u+1\right)= \nonumber\\
&=\sum_{j=0}^{N-1}\varphi(j)\left(\mathbb{P}(S_N\leqslant N-j-1)+\sum_{u=0}^{v}\mu_j(u)\right).
\end{align}
As the survival probability $\varphi(u)$ is non-decreasing, there exists non-negative limit
$\varphi(\infty):=\lim_{u\rightarrow\infty}\varphi(u)$.
By letting $v\to\infty$, from the equality \eqref{ga}, we get
$$
\varphi(\infty)\big(N-\mathbb{E}S_N\big)\geqslant 0.
$$
See, \cite[p. 4]{GS} or \cite[p. 13]{GS1} for the limit assesment of the left-hand side of \eqref{ga}.
The last inequality, together with condition $\mathbb{E}\,S_N>N$, implies $\varphi(\infty)=0$, and consequently $\varphi(u)=0$ for all $u\in\mathbb{N}_0$.

    We now suppose that $\mathbb{E}\, S_N = N$ and  $\mathbb{P}(S_N=N)<1$. In this case, there is at least one positive probability out of $s_0^{(N)},\, s_1^{(N)},\, s_2^{(N)},\,\ldots,\, s_{N-1}^{(N)}$.  Since $\mathbb{E} S_N = N$, by letting $v\to\infty$ in equation \eqref{ga}, we get
\begin{align}\label{gaga}
\sum_{j=0}^{N-1}\varphi(j)\left(\sum_{u=0}^{N-j-1}s_u^{(N)}+\sum_{u=0}^{\infty}\mu_j(u)\right)=0.
\end{align}

If $s_0^{(N)}>0$, then equation (\ref{gaga}) implies $\varphi(0)=\varphi(1)=\ldots=\varphi(N-1)=0$, and the main recursive relation \eqref{2+} implies $\varphi(u)=0$ for all $u\in\mathbb{N}_0$. If $s_0^{(N)}=0$ and $s_1^{(N)}>0$, then equation \eqref{gaga} implies $\varphi(0)=\varphi(1)=\ldots=\varphi(N-2)=0$, and the  recursive relation \eqref{2+} implies again $\varphi(u)=0$ for all $u\in\mathbb{N}_0$. Proceeding further we consider the cases $s_0^{(N)}=0,\,\ldots,\, s_k^{(N)}=0,\, s_{k+1}^{(N)}>0$ far all $k=2,\,3,\,\ldots,\, N-2$, and, using \eqref{gaga} and \eqref{2+}, we obtain $\varphi(u)=0$ for all $u\in\mathbb{N}_0$. This finishes the proof of part (i) of Theorem \ref{t2}. We note that the assertion of part (i) can be proved by using some different method present in \cite[p. 10]{GS} or \cite[p. 21]{GS1}.

If $\mathbb{E} S_N=N$ and $\mathbb{P}(S_N=N)=1$, then random variables $Z_1,\,Z_2,\,\ldots,\,Z_N$ are degenerate, i.e. $Z_k\equiv i_k$ for some $i_k\in\{0,\,1,\,\ldots,\,N\}$ and for all $k\in\{1,\,2,\,\ldots,\,N\}$. Therefore, in the case under consideration the \textit{DTRM(N}), defined by \eqref{pop}, becomes deterministic. To complete the proof of assertion (ii) of Theorem \ref{t2} it suffices to find $u\geqslant0$ such that $W_u(n)>0$ for $n=0,\,1,\,\ldots,\,N$.
\end{proof}

\section{Proof of Theorem \ref{t3}}\label{pt3}

The assertion of Theorem \ref{t3} follows from the more general statement - Theorem \ref{t1}. On the other hand, the bi-seasonal model when $N=2$ is the most simple one aiming to demonstrate everything in detail. We next prove the statement of Theorem \ref{t3}.

\begin{proof}[Proof of Theorem \ref{t3}]
Let us begin with the case (i). If $\alpha\in\mathbb(-1,0)$ is a unique root of equation $G_{S_2}(s)=s^2$ (see \cite[Cor. 15]{GJ}), then, by \eqref{aiai+},
\begin{align*}
m_0^{(1)}z_0^{(2)}+m_0^{(2)}z_0^{(1)}\frac{G_{Z_2}(\alpha)}{\alpha}=0,
\end{align*}
and \eqref{aiai+++} implies
\begin{align}\label{ty}
m_0^{(1)}z_0^{(2)}+m_0^{(2)}z_0^{(1)}=2-\mathbb{E}S_2.
\end{align}
These two equalities together with expression \eqref{cai} give the second expression of the part (i)
\begin{align}\label{phi1_N2}
\varphi(1)=m_0^{(1)}=\frac{2-\mathbb{E}S_2}{z_0^{(2)}}\,\frac{G_{Z_2}(\alpha)}{G_{Z_2}(\alpha)-\alpha}.
\end{align}
The first equation of the system \eqref{aiai++++} implies
\begin{align}\label{ta}
m_1^{(1)}z_0^{(2)}+m_0^{(1)}z_1^{(2)}+m_0^{(1)}z_0^{(2)}=m_0^{(2)}.
\end{align}
Equalities \eqref{ty} and \eqref{ta} imply
\begin{align*}
\big(m_0^{(1)}+m_1^{(1)}\big)z_0^{(1)}z_0^{(2)}=2-\mathbb{E}S_2-m_0^{(1)}\big(z_0^{(2)}-z_0^{(1)}z_1^{(2)}\big).
\end{align*}
According to formula \eqref{cai}, $\varphi(2)=m_0^{(1)}+m_1^{(1)}$ and $\varphi(1)=m_0^{(1)}$. Therefore,
\begin{align*}
z_0^{(1)}z_0^{(2)}\varphi(2)=2-\mathbb{E}S_2-\varphi(1)\big(z_0^{(2)}-z_0^{(1)}z_1^{(2)}\big).
\end{align*}
On the other hand, the main recursive relation \eqref{2+} gives that
\begin{align*}
\varphi(0)=z_0^{(1)}z_1^{(1)}\varphi(1)+z_0^{(1)}z_0^{(2)}\varphi(2).
\end{align*}
Consequently,
\begin{align}\label{tu}
\varphi(0)+z_0^{(2)}\varphi(1)=2-\mathbb{E}S_2.
\end{align}
Substituting the expression of (\ref{phi1_N2}) into the last equality (\ref{tu}) we obtain the desired expression of $\varphi(0)$. Part (i) of the theorem is proved.

We now consider the part (ii) of Theorem \ref{t3}. Since $z_0^{(1)}>0$ and $z_0^{(2)}=0$, equality \eqref{tu} implies immediately that $\varphi(0)=2-\mathbb{E}S_2$. On the other hand, equalities \eqref{ty} and \eqref{ta} give that
\begin{align*}
m_0^{(1)}z_1^{(2)}=m_0^{(2)}\ {\rm and}\  m_0^{(2)}z_0^{(1)}=2-\mathbb{E}S_2.
\end{align*}
Therefore,
\begin{align*}
\varphi(1)=m_0^{(1)}=\frac{m_0^{(2)}}{z_1^{(2)}}=\frac{2-\mathbb{E}S_2}{z_0^{(1)}z_1^{(2)}}
\end{align*}
and part (ii) is proved.

Consider the part (iii) of Theorem \ref{t3}. Conditions $z_0^{(1)}=0,\, z_0^{(2)}>0$ and equality \eqref{ty} imply
\begin{align*}
\varphi(1)=m_0^{(1)}=\frac{2-\mathbb{E}S_2}{z_0^{(2)}}.
\end{align*}
Further, in view of the last expression of $\varphi(1)$, the equality \eqref{tu} gives $\varphi(0)=0$. This is natural because the survival is impossible if the claim size at the first moment of time is at least one and the initial surplus equals zero. The remaining part for $\mathbb{E}S_2\geqslant2$ is implied by Theorem \ref{t2}. Theorem \ref{t3} is proved.
\end{proof}

\section{Guidelines for the survival probability calculation}\label{alg}

In this section we develop an algorithm for the survival probability calculation of the multi-seasonal discrete time risk model $DTRM(N)$. The case $N=2$ is completely described in Theorem \ref{t3}. If $N\geqslant 3$, the survival probability calculation algorithm is more complex. Below we provide main steps for such calculation. 
\textit{\textbf{Step A}}. In cases where $\mathbb{E}S_N=\sum_{k=1}^N\mathbb{E}Z_k>N$ or $\mathbb{E}S_N=N$ with $\mathbb{P}(S_N=N)<1$, then, according to parts (i) and (ii) of Theorem \ref{t2}, we have $\varphi(u)=0$ for all $u\in\mathbb{N}_0$.

\textit{\textbf{Step B}}. If $\mathbb{E}S_N=N$ and $\mathbb{P}(S_N=N)=1$, then random variables $Z_1,\, Z_2,\,\ldots,\, Z_N$  are degenerate, and $\varphi(u)$ equals zero or one according to the part (iii) of Theorem \ref{t2}.

\textit{\textbf{Step C}}. If the net profit condition $\mathbb{E}S_N<N$ holds then the survival probabilities $\varphi(1),\, \varphi(2),\, \ldots,\, \varphi(N+1)$ are calculated using the formula \eqref{cai}. First, we solve required probabilities $m_0^{(1)},\, m_1^{(1)},\, \ldots,\, m_N^{(1)}$ from the systems of equations in Theorem \ref{t1}. That is accomplished the following way:

\textbf{(i)} Find the roots of the equation $G_{S_N}(s)=s^N$ satisfying the condition $|s|<1$.

\begin{remark}\label{rem} \textit{In view of the well known Rouch\'e theorem (see, for instance, Chapter 10 in \cite{Rudin}) and estimate $|G_{S_N}(s)|<|\lambda s^k|$ when $\lambda>1$ and $|s|=1$, the functions $\lambda s^N$ and $G_{S_N}(s)-\lambda s^N$ have the same number of roots in the region $|s|<1$, and this number is $N$ due to $\lambda s^N$. As $\lambda \to 1^+$, the equation $G_{S_N}(s)=s^N$ has at least $N$ solutions in $|s|\leqslant1$, counted with their multiplicities. One of them always occurs at $s=1$ and its multiplicity is equal to one if the net profit condition $G^\prime_{S_N}(1)=\mathbb{E}S_N<N$ holds. Under general conditions, the remaining $N-1$ roots of the equation $G_{S_N}(s)=s^N$ lie inside the unit circle $|s|<1$, see Section 4 in \cite{GJ}.
One may also observe that a complex roots of $G_{S_N}(s)=s^N$ occur in conjugate pairs due to $G_{S_N}(\overline{s})-\overline{s}^N=\overline{G_{S_N}(s)-s^N}$, where the over-line denotes the complex conjugate.
}
\end{remark}

\textbf{(ii)} Having roots $\alpha_1,\,\alpha_2,\,\ldots,\, \alpha_{N-1}$ of the equation $G_{S_N}(s)=s^N, |s|<1,$ we use equations \eqref{aiai+++}, \eqref{aiai+}, \eqref{aiai++} and \eqref{aiai++++} of Theorem \ref{t1} to get $m_0^{(1)},\,m_0^{(2)},\,\ldots,\,m_0^{(N)}$. If probability $s_0^{(N)}$ is positive and all the roots $\alpha_1,\,\alpha_2,\,\ldots,\, \alpha_{N-1}$ are of multiplicity one, then we set up the following system
\begin{tiny}%
\begin{multline}\label{NN}
\begin{pmatrix}
z_0^{(N)}&\frac{z_0^{(1)}G_{Z_N}(\al_1)}{\al_1}&\frac{z_0^{(2)}G_{Z_N+Z_1}(\al_1)}{\al^2_1}&\ldots&
\frac{z_0^{(N-1)}G_{Z_N+Z_1+\ldots+Z_{N-2}}(\al_1)}{\al^{N-1}_1}\\
z_0^{(N)}&\frac{z_0^{(1)}G_{Z_N}(\al_2)}{\al_2}&\frac{z_0^{(2)}G_{Z_N+Z_1}(\al_2)}{\al^2_2}&\ldots&
\frac{z_0^{(N-1)}G_{Z_N+Z_1+\ldots+Z_{N-2}}(\al_2)}{\al^{N-1}_2}\\
\vdots&\vdots &\vdots & \ddots&\vdots\\
z_0^{(N)}&\frac{z_0^{(1)}G_{Z_N}(\al_{N-1})}{\al_{N-1}}&\frac{z_0^{(2)}G_{Z_N+Z_1}(\al_{N-1})}{\al^2_{N-1}}&\ldots&
\frac{z_0^{(N-1)}G_{Z_N+Z_1+\ldots+Z_{N-2}}(\al_{N-1})}{\al^{N-1}_{N-1}}\vspace{2mm}\\
z_0^{(N)}&z_0^{(1)}&z_0^{(2)}&\ldots&z_0^{(N-1)}
\end{pmatrix}
\begin{pmatrix}
m_0^{(1)}\\
m_0^{(2)}\\
\vdots\\
m_0^{(N-1)}\\
m_0^{(N)}
\end{pmatrix}
= \\=
\begin{pmatrix}
0\\
0\\
\vdots\\
0\\
N-\mathbb{E}S_N
\end{pmatrix}.
\end{multline}
\end{tiny}%


\textbf{ (iii)} If some of the roots of equation $G_{S_N}(s)=s^N$ are of multiplicity higher than one but the probability $s^{(N)}_0>0$, we replace the corresponding line (or lines) in the main matrix of system  (\ref{NN}) by line (or lines) with coefficients based on the corresponding equality from the set of equalities \eqref{aiai++}. 

\textbf{(iv)} If $s^N_0=0$ then the main matrix in \eqref{NN} is singular because some of $z_0^{(1)},\, z_0^{(2)},\,\ldots,\,z_0^{(N)}$ are zeros (c.f. cases {\rm (ii)} and {\rm (iii)} in Theorem \ref{t3}). These probabilities cannot all equal to zero as this would violate the net profit condition $\mathbb{E}S_N<N$. If some of $z_0^{(1)},\, z_0^{(2)},\,\ldots,\,z_0^{(N)}$ vanish, then we reconstruct the system \eqref{NN} and observe that the system \eqref{aiai++++} provides additional relations between probabilities  $m^{(k)}_{n-1}$ and $z^{(k)}_{n-1}$ for  $k\in \{1,\,2,\,\ldots,\,N\}$ and  $n\in\mathbb{N}$.

\begin{remark}
\textit{In particular, for $n=1$, \eqref{aiai++++} provides the missing information on some of those $m_0^{(1)},\,m_0^{(2)},\,\ldots,\,m_0^{(N)}$ which are not be given by this reconstructed system \eqref{NN} if some $z_0^{(1)},\,z_0^{(2)},\,\ldots,\,z_0^{(N)}$ equal zero. Indeed, for $n=1$, from
 \eqref{aiai++++} we have
 \begin{tiny}%
\begin{align}\label{NN1}
\begin{pmatrix}
z^{(N)}_0&0&\ldots&0\\
0&z^{(1)}_0&\ldots&0\\
\vdots&\vdots&\ddots&\vdots\\
0&0&\ldots&z_0^{(N-1)}
\end{pmatrix}
\begin{pmatrix}
m^{(1)}_1\\
m^{(2)}_1\\
\vdots\\
m^{(N)}_1
\end{pmatrix}
=
\begin{pmatrix}
m^{(N)}_0-\big(z^{(N)}_0+z^{(N)}_1\big)m_0^{(1)}\\
m^{(1)}_0-\big(z^{(1)}_0+z^{(1)}_1\big)m^{(2)}_0\\
\vdots\\
m^{(N-1)}_0-\big(z_0^{(N-1)}+z_1^{(N-1)}\big)m^{(N)}_0
\end{pmatrix}
.
\end{align}
\end{tiny}%
For some probabilities out of $z_0^{(1)},\, z_0^{(2)},\,\ldots,\,z_0^{(N)}$ being equal to zero, the corresponding terms on the left-hand side of the system \eqref{NN1} are equal to zero too, and the corresponding probabilities $m_0^{(1)},\,m_0^{(2)},\,\ldots,\,m_0^{(N)}$ can be solved.}
 \end{remark}

\textbf{(v)} Under the condition that the system of linear equations obtained in the previous steps (ii) and (iii) has a solution $m_0^{(1)},\,m_0^{(2)},\,\ldots,\,m_0^{(N)}$, the further probabilities $m_n^{(1)},\,m_n^{(2)},\,\ldots,\,m_n^{(N)}$, $n=1,\,2,\,\ldots,\,N$ that are needed for the formula \eqref{cai}, are calculated from the system \eqref{aiai++++}.

\textit{\textbf{Step D}}. By the main recursive formula \eqref{2+} we obtain $\varphi(u)$ for $u=0$ and for all $u\geqslant N+1$.

\section{Numerical examples}\label{nuex}

In this section, we present four survival probability calculation examples. All calculations were performed according to the algorithms from section \ref{alg} and  were carried out by program Wolfram Mathematica \cite{Mathematica}. We recall  that a random variable  $X$ is distributed according to the  Poisson law  with parameter $\lambda>0$, and denote this by $X\sim\mathcal{P}(\lambda)$, if $$\mathbb{P}(X=k)=e^{-\lambda}\lambda^k/k!,\  k\in\mathbb{N}_0.$$

\begin{example}
\textit{Let the bi-risk model ($DTRM(2)$) be generated by random variables $Z_1\sim\mathcal{P}(0.3)$ and $Z_2\sim\mathcal{P}(1.4)$.
We find the initial value of the survival probability $\varphi(0)$.}
\end{example}

For the considered example, we have $\mathbb{E}S_2=\mathbb{E}\left(Z_1+Z_2\right)=1.7<2$. Therefore, according to the algorithm, described in Section \ref{alg}, we should construct generating function of sum $S_2=Z_1+Z_2$. In view of the generating function ${\rm e}^{\lambda\,(s-1)}$ of the Poisson law  $\mathcal{P}(\lambda)$, we have
$$
G_{S_2}(s)=G_{Z_1}(s)\,G_{Z_2}(s)={\rm e}^{1.7\,(s-1)}, \, s\in\mathbb{C}.
$$
Equation ${\rm e}^{1.7(s-1)}=s^2$ has a root $\alpha=-0.3244096519$. Therefore, due to the part (i) of Theorem \ref{t3}
$$
\varphi(0)=(2-\mathbb{E}S_2)\,\frac{\alpha}{\alpha-G_{Z_2}(\alpha)}=0.3\,\frac{\alpha}{\alpha-{\rm e}^{1.4(\alpha-1)}}=0.2023378868,
$$
which coincides with the initial value of $\varphi(0)$ obtained  in \cite{DS} by different method.

\begin{example}
\textit{Let the DTRM(3) be generated by random variables  $Z_1\sim\mathcal{P}(1/2)$, $Z_2\sim\mathcal{P}(2/3)$ and $Z_3\sim\mathcal{P}(4/5)$. We find $\varphi(1)$ and $\varphi(2)$.}
\end{example}

Since $\mathbb{E}S_3=59/30<3$, we should solve the equation
$$
G_{S_3}(s)={\rm e}^{(s-1)59/30}=s^3, \, |s|<1,
$$
first. This equation
has the following two complex conjugate solutions
\begin{align*}
\al_1: = -0.287678 - 0.319495 {\rm i}\ \text{ and }\  \al_2= -0.287678 + 0.319495 {\rm i},
\end{align*}
Setting these values into the system \eqref{NN} for $N=3$  we obtain $\big(m_0^{(1)},\,m_0^{(2)},\,m_0^{(3)}\big)=(0.699796,\, 0.644968,\, 0.638276)$.
Cosequently, $m_0^{(1)}=\varphi(1)=0.699796$ and, according to the system \eqref{NN1}, $m_0^{(1)}+m_1^{(1)}=\varphi(2)= 0.860672$, which coincide with the corresponding values  of $\varphi(1)$ and $\varphi(2)$ given in \cite{GKS}.

\begin{example}
\textit{Let the DTRM(3) be generated by random variables  $Z_1$, $Z_2$ and $Z_3$, where
$$\mathbb{P}(Z_1=0)=0.8=1-\mathbb{P}(Z_1=1),$$ $$\mathbb{P}(Z_2=0)=0.2=1-\mathbb{P}(Z_2=1),\ \,   Z_3\mathop{=}\limits^{d} Z_1.$$
We find the survival probability $\varphi(u)$ for all $u\in\mathbb{N}_0$.}
\end{example}

Since $\mathbb{E}S_3=1.2<3$ again, we solve the equation
$
G_{S_3}(s)=s^3
$,
which is
 $$
 (0.8+0.2s)^2(0.2+0.8s)=s^3,\, |s|<1,
 $$
and has a root $s=-4/11$ of multiplicity two. According to the guidelines given in Step 3 (iii) in Section \ref{alg}, we should set up the modified system \eqref{NN} to obtain $m_0^{(1)}$. Such system is 
$$
\begin{pmatrix}
0.8&-1.6&0.8\\
0&-4.84&4.84\\
0.8&0.8&0.2
\end{pmatrix}
\begin{pmatrix}
m_0^{(1)}\\
m_0^{(2)}\\
m_0^{(3)}
\end{pmatrix}
=
\begin{pmatrix}
0\\
0\\
1.8
\end{pmatrix}.
$$
It implies $\varphi(1)=m_0^{(1)}=1$, and, consequently, $\varphi(u)=1$ for $u\in\{2,\,3,\,\ldots\}$. Finally, from the main recursive relation \eqref{2+}  we  conclude that $\varphi(0)=0.8$.

\begin{example}
\textit{Let us consider the DTRM(10) generated by ten Poissonian random variables $Z_k\sim\mathcal{P}(k/(k+1))$, $k\in\{1,\,2,\,\ldots,\,10\}$. We calculate $\varphi(u)$ for $u\in\{0,\,1,\,\ldots,\,15\}$ and also include the finite time survival probability
$$
\varphi(u,T):=\mathbb{P}\left(\bigcap_{n=1}^{T}\left\{u+n-\sum_{k=1}^{n}Z_k>0\right\}\right),
$$
when $u\in\{0,\,1,\,\ldots,\,15\}$ and $T\in\{1,\,2\,\ldots,\,15\}$.}
\end{example}

Arguing the same as deriving the main recursive relation for survival probability \eqref{2+}, for the $N$-seasonal model we have that
\begin{align*}
&\varphi(u,1)=\sum_{k\leqslant u}z^{(1)}_{k},\quad
\varphi(u,2)=\hspace{-2mm}\sum_{\substack {k_1\leqslant u\\k_2\leqslant u+1-k_1}}\hspace{-2mm}z^{(1)}_{k_1}z^{(2)}_{k_2},\quad
\ldots,\\
&\varphi(u,N)=\hspace{-6mm}\sum_{\substack{k_1\leqslant u\\k_2\leqslant u+1-k_1\\ \ldots\\k_N\leqslant u+N-1-k_1-\ldots-k_{N-1}}} \hspace{-8mm}z^{(1)}_{k_1}z^{(2)}_{k_2}\ldots z^{(N)}_{k_N},
\end{align*}
and
\begin{align*}
\varphi(u,T)=\hspace{-6mm}\sum_{\substack{k_1\leqslant u\\k_2\leqslant u+1-k_1\\ \ldots\\k_N\leqslant u+N-1-k_1-\ldots-k_{N-1}}}\hspace{-8mm}z^{(1)}_{k_1}z^{(2)}_{k_2}\cdots z^{(N)}_{k_N}\varphi\big(u+N-k_1-\ldots-k_N,\,T-N\big)
\end{align*}
if $T\in\{N+1,\, N+2,\,\ldots\}$. On the other hand, the finite time ruin probability, even more efficiently in terms of computational time, can be calculated by using an algorithm present in \cite{BBS}.

 Using the obtained formulas of $\varphi(u,T)$ and algorithm given in Section \ref{alg}, we calculate  the required values of $\varphi(u)$ and $\varphi(u,T)$.
 In Table \ref{tab} below we list the finite and ultimate time survival probabilities of the $10$-seasonal model. Numbers are rounded up to three decimal places except when the rounding result is $1$.

\begin{table}[H]
\centering
\caption{Survival probabilities of $10$-seasonal model.}\label{tab}
\vspace{2mm}
\begin{tabular}{|c|c|c|c|c|c|c|c|c|}
\hline
{$T$}&{$u=0$}&{$u=1$}&{$u=2$}&{$u=3$}&{$u=4$}&{$u=5$}&{$u=10$}&{$u=15$}\\
\hline
$1$&$0.607$&$0.910$&$0.986$&$0.998$&$1$&$1$&$1$&$1$\\
$2$&$0.519$&$0.848$&$0.963$&$0.992$&$0.999$&$1$&$1$&$1$\\
$3$&$0.470$&$0.801$&$0.938$&$0.983$&$0.996$&$0.999$&$1$&$1$\\
$4$&$0.437$&$0.763$&$0.914$&$0.972$&$0.991$&$0.998$&$1$&$1$\\
$5$&$0.412$&$0.732$&$0.891$&$0.959$&$0.986$&$0.995$&$1$&$1$\\
$10$&$0.339$&$0.626$&$0.798$&$0.894$&$0.947$&$0.975$&$1$&$1$\\
$15$&$0.319$&$0.595$&$0.766$&$0.868$&$0.928$&$0.962$&$0.999$&$1$\\
\hline
$\infty$&$0.284$&$0.535$&$0.698$&$0.803$&$0.871$&$0.916$&$0.990$&$0.999$\\
\hline
\end{tabular}
\end{table}

The provided Table \ref{tab} suggests that a "comfortable" long term survival conditions are met when initial surplus $u\geqslant 15$ under the case of consideration.

\end{document}